\documentclass[10pt,reqno,a4paper]{amsart}

\usepackage{euscript}
\usepackage{xcolor}

\newtheorem{theorem}{Theorem}

\newtheorem{lemma}{Lemma}

\theoremstyle{definition}
\newtheorem{definition}{Definition}[section]
\newtheorem{example}{Example}[section]
\newtheorem{remark}{Remark}[section]

\renewcommand{\epsilon}{\varepsilon}

\def\Id{\text{\rm Id}}
\def\cA{\EuScript{A}}
\def\N{\mathbb{N}}
\def\Z{\mathbb{Z}}
\def\R{\mathbb{R}}

\begin{document}
\title[A generalized Grobman-Hartman]{A generalized Grobman-Hartman theorem for nonautonomous dynamics}

\author{Lucas Backes}
\address{\noindent Departamento de Matem\'atica, Universidade Federal do Rio Grande do Sul, Av. Bento Gon\c{c}alves 9500, CEP 91509-900, Porto Alegre, RS, Brazil.}
\email{lucas.backes@ufrgs.br} 

\author{Davor Dragi\v cevi\' c}
\address{Department of Mathematics, University of Rijeka, Croatia}
\email{ddragicevic@math.uniri.hr}

\begin{abstract}
The purpose of this note is to extend the recent generalized version of the Grobman-Hartman theorem established by Bernardes Jr. and  Messaoudi from  an autonomous to nonautonomous dynamics. More precisely, we prove that any sufficiently small perturbation of a nonautonomous linear dynamics 
that admits a generalized exponential dichotomy is topologically conjugated to its linear part. In addition, we prove that under certain mild additional conditions, the conjugacy is in fact H\"older continuous.
\end{abstract}

\keywords{generalized exponential dichotomy, linearization, H\"older conjugacy}
\subjclass[2010]{34D30, 34D09}
\maketitle

\section{Introduction}
The classical Grobman-Hartman theorem~\cite{G1,G2, H1,H2} is one of the most celebrated results in the qualitative theory of differential equations and dynamical systems. It asserts that for a hyperbolic linear automorphism $A$ on $\R^d$ (i.e. $A$ is an invertible operator whose spectrum doesn't intersect the unit circle in $\mathbb C$) and 
for any bounded Lipschitz map $f\colon \R^d \to \R^d$ whose Lipschitz constant is sufficiently small, there exists a homeomorphism $h\colon \R^d \to \R^d$ such that $h\circ A=(A+f)\circ h$. This result was extended to Banach spaces independently by Palis~\cite{Palis} and Pugh~\cite{Pugh} 
(who also simplified the original arguments of Grobman and Hartman). It is well-known that the conjugacy $h$ is in general only H\"older continuous. Indeed, although apparently this fact was known to experts for some time, the first rigorous proof was published by Shi and Xiong~\cite{SX}.
Nevertheless, many works were devoted to the problem of formulating sufficient conditions which would ensure that the conjugacy $h$ exhibits higher regularity. In this direction, we mention the seminal works of Sternberg~\cite{Stern} and Belitskii~\cite{Bel73,Bel78}
as well as some more recent contributions~\cite{E1,E2, R-S-JDDE04, R-S-JDDE06, ZhangZhang14JDE,ZZJ}.

We stress that all the above mentioned works deal with the case when $A$ is a hyperbolic operator. Recently, Bernardes Jr. and  Messaoudi~\cite{B} showed that the conclusion of the Grobman-Hartman theorem holds true under weaker assumption that $A$ is a \emph{generalized hyperbolic operator}. This weaker notion of 
hyperbolicity was introduced and studied by Cirilo et.al. in~\cite{C}. As in the classical notion of hyperbolicity, the notion of generalized hyperbolicity requires that the domain of $A$ splits into two closed subspaces, one of which is contracting while the other is expanding under the action of $A$. However, unlike what happens in the hyperbolic case, these
 subspaces don't need to be invariant with respect to $A$ (see Example~\ref{ex: ghyp op} for details).

We emphasize that so far we discussed only the case of autonomous dynamics. The first version of the Grobman-Hartman theorem for nonautonomous dynamics with continuous time was established by Palmer~\cite{Palmer}. The case of nonautonomous dynamics with discrete time was first considered by
Aulbach and Wanner~\cite{AW}. Since then many authors have obtained valuable contribution to nonautonomous linearization (see for example~\cite{Jiang, Lin, LFP, RS, SX} and references therein). We particularly mention the recent results~\cite{CR,CDS, DZZ, DZZ2} dealing with higher regularity of the conjugacies (which as in the autonomous case are in general only
H\"older continuous).

The main objective of the present paper is to obtain a nonautonomous version of the generalized Grobman-Hartman theorem established in~\cite{B}. More precisely, we introduce the notion of a generalized exponential dichotomy which extends the classical notion of exponential dichotomy and in addition, when restricted 
to the autonomous case coincides with the notion of a generalized hyperbolic operator. We then prove that any sufficiently small nonlinear perturbation of a linear dynamics that admits a generalized exponential dichotomy is topologically conjugated to the linear part. In addition, we prove that  conjugacies are H\"older continuous. We emphasize that our results (and their proofs) are inspired by those in~\cite{B}.

The paper is organized as follows. In Section~\ref{P} we introduce the notion of a generalized exponential dichotomy and make several important remarks related to it. In Section~\ref{M1} we established the main result of our paper. Namely, we prove the above mentioned generalized version of the Grobman-Hartman theorem for nonautonomous dynamics. Then, in Section~\ref{M2} we prove that the conjugacies are H\"older continuous under certain additional assumptions. Finally, in Section~\ref{sec: examples} we construct some explicit examples of nonautonomous dynamics that admits a generalized exponential dichotomy.

\section{Preliminaries}\label{P}
Let $X=(X, \lVert \cdot \rVert)$ be an arbitrary Banach space and denote by $\mathcal B(X)$ the space of all bounded linear operators on $X$.  Given a sequence $(A_n)_{n\in \Z}$ of invertible operators in $\mathcal B(X)$, we define the associated \emph{linear cocycle} by
\[
\cA(m, n)=\begin{cases}
A_{m-1}\cdots A_n & \text{if $m>n$;}\\
\Id & \text{if $m=n$;} \\
A_m^{-1}\cdots A_{n-1}^{-1} & \text{if $m<n$.}
\end{cases}
\]

\subsection{Generalized exponential dichotomy} We now introduce the main concept that we are going to consider in this paper. Namely, we introduce the notion of a generalized exponential dichotomy.  This notion  is a generalization of the notion of generalized hyperbolic operator (see Example \ref{ex: ghyp op}) introduced in \cite{C} (and further studied in~\cite{B}) to the nonautonomous setting.
\begin{definition}\label{def: gexp}
Let $(A_n)_{n\in \Z}$ be a sequence of invertible operators in $\mathcal B(X)$. We say that $(A_n)_{n\in \Z}$ admits a \emph{generalized exponential dichotomy} if:
\begin{itemize}
\item for each $n\in \Z$ there are closed subspaces $S(n)$ and $U(n)$ of $X$ such that
\begin{equation}\label{split}
X=S(n)\oplus U(n) \quad \text{for $n\in \Z$;}
\end{equation}
\item for each $n\in \Z$,
\begin{equation}\label{inv}
A_n S(n)\subset S(n+1) \quad \text{and} \quad A_n^{-1} U(n+1)\subset U(n);
\end{equation}
\item there exist $D, \lambda >0$ such that 
\begin{equation}\label{e1}
\lVert \cA(m, n)x\rVert \le De^{-\lambda (m-n)}\lVert x\rVert  \quad \text{for $x\in S(n)$ and $m\ge n$,}
\end{equation}
and
\begin{equation}\label{e2}
\lVert \cA(m, n)x\rVert \le De^{-\lambda (n-m)}\lVert x\rVert \quad \text{for $x\in U(n)$ and $m\le n$;}
\end{equation}
\item we have that
\begin{equation}\label{bp}
\sup_{n\in \Z} \lVert P_n\rVert <\infty, 
\end{equation}
where $P_n \colon X\to S(n)$ is a projection associated with the decomposition~\eqref{split}.
\end{itemize}
\end{definition}

Let us make some observations about this definition. 
\begin{remark}
We observe that it follows easily from~\eqref{e1}, \eqref{e2} and~\eqref{bp} that, by increasing $D$ if necessary, we have 
\begin{equation}\label{e11}
\lVert \cA(m, n)P_n\rVert \le De^{-\lambda (m-n)} \quad \text{for $m\ge n$,}
\end{equation}
and
\begin{equation}\label{e22}
\lVert \cA(m, n)(\Id-P_n)\rVert \le De^{-\lambda (n-m)} \quad \text{for $m\le n$.}
\end{equation}
We are going to use this simple observation  in the sequel.
\end{remark}

\begin{remark}
The notion of a generalized exponential dichotomy is similar to the classical notion of an exponential dichotomy (see~\cite{Coppel, H}). The important difference is that in the notion of an exponential dichotomy it is required that
\[
A_n S(n)=S(n+1) \quad \text{and} \quad A_n^{-1} U(n+1)= U(n),
\]
which is obviously a stronger requirement than~\eqref{inv}.  
We refer to Section \ref{sec: examples} for several examples of nonautonomous dynamics that admits a generalized exponential dichotomy but doesn't admit an exponential dichotomy. 
\end{remark}

\begin{remark}
In the case when $X$ is finite-dimensional, the notion of a generalized exponential dichotomy reduces to the notion of \emph{exponential trichotomy} introduced in~\cite{EH,P} (see also~\cite{BD}). Indeed, let $(y_n)_{n\in \Z}\subset X$ be such that 
$\sup_{n\in \Z}\lVert y_n\rVert <\infty$. Then, it follows from~\eqref{e11} and~\eqref{e22} that the sequence $(x_n)_{n\in \Z}\subset X$ defined by
\[
x_n=\sum_{k=-\infty}^n \cA(n, k)P_k y_k -\sum_{k=n+1}^\infty \cA(n, k)(\Id-P_k)y_k \quad n\in \Z,
\]
satisfies $\sup_{n\in \Z}\lVert x_n\rVert<\infty$. In addition, it is easy to verify that
\[
x_{n+1}-A_n x_n=y_{n+1}, \quad \text{for $n\in \Z$.}
\]
It follows now from~\cite[Proposition 1.]{P} that $(A_n)_{n\in \Z}$ admits an exponential trichotomy. 
\end{remark}

\section{A generalized nonautonomous Grobman-Hartman theorem}\label{M1}

We now establish the main result of this paper. This can be regarded as a nonautonomous version of the generalized Grobman-Hartman theorem established in~\cite{B}.
\begin{theorem}\label{theo: GH}
Assume that $(A_n)_{n\in \Z}$ is a sequence of invertible operators in $\mathcal B(X)$ that admits a generalized exponential dichotomy. Furthermore, let $(f_n)_{n\in \Z}$ be a sequence of maps $f_n \colon X\to X$ such that:
\begin{enumerate}
\item there exist $c>0$ such that
\begin{equation}\label{f}
\lVert f_n(x)-f_n(y)\rVert \le c\lVert x-y\rVert, \quad \text{for $n\in \Z$ and $x, y\in X$;}
\end{equation}
\item $A_n+f_n$ is a homeomorphism for each $n\in \Z$;
\item \begin{equation}\label{b}
\sup_{n\in \Z} \lVert f_n\rVert_\infty<+\infty, 
\end{equation}
where 
\[
\lVert f_n\rVert_\infty:=\sup \{\lVert f_n (x)\rVert: x\in X\}.
\]
\end{enumerate}
Then, if $c$ is sufficiently small, there exists a sequence $(H_n)_{n\in \Z}$ of homeomorphisms on $X$ such that 
\begin{equation}\label{lin}
H_{n+1}\circ A_n= (A_n+f_n)\circ H_n \quad \text{for $n\in \Z$.}
\end{equation}
In addition, 
\[
\sup_{n\in \Z}\lVert H_n-\Id \rVert_\infty <+\infty. 
\]
Furthemore, for each $n\in \Z$ and $x\in X$, $H_n(x)-x\in S(n)+A_n^{-1}U(n+1)$. Finally, the sequence $(H_n)_{n\in \Z}$ with the above properties is unique. 
\end{theorem}

\begin{remark}\label{rem: A+f homeo}
We recall (see~\cite[p. 433]{BD0} or the arguments in Example~\ref{ex: unbounded}) that if  $a=\sup_{n\in \Z} \lVert A_n^{-1}\rVert<\infty$ and $ac<1$, that then $A_n+f_n$ is a homeomorphism for each $n\in \Z$.
\end{remark}

\begin{proof}[Proof of Theorem~\ref{theo: GH}]
We define $\mathcal Y$ to be the space which consists of all two-sided sequences  $\mathbf h=(h_n)_{n\in \Z}$ of continuous maps on $X$ such that:
\begin{itemize}
\item for $n\in \Z$ and $x\in X$, $h_n(x)\in S(n)+A_n^{-1}U(n+1)$;
\item 
\[
\lVert \mathbf h\rVert_{\mathcal Y} :=\sup_{n\in \Z} \lVert h_n\rVert_\infty <+\infty. 
\]
\end{itemize}
It is easy to verify that $(\mathcal Y, \lVert \cdot \rVert_{\mathcal Y})$ is a Banach space.  We define $\mathcal T\colon \mathcal Y \to \mathcal Y$ by 
\[
\begin{split}
&(\mathcal T\mathbf h)_n (x) \\
&=\sum_{k=-\infty}^n \cA(n, k)P_k (f_{k-1}(\cA(k-1, n)x+h_{k-1}(\cA(k-1, n)x))) \\
&\phantom{=}-\sum_{k=n+1}^\infty \cA(n, k)(\Id-P_k)(f_{k-1}(\cA(k-1, n)x+h_{k-1}(\cA(k-1, n)x))),
\end{split}
\]
for $n\in \Z$, $x\in X$ and  $\mathbf h=(h_n)_{n\in \Z}\in \mathcal Y$. Observe that~\eqref{e11} and~\eqref{e22} imply that 
\[
\begin{split}
\lVert (\mathcal T\mathbf h)_n \rVert_\infty &\le \sum_{k=-\infty}^n De^{-\lambda (n-k)}\lVert f_{k-1}\rVert_\infty +\sum_{k=n+1}^\infty De^{-\lambda (k-n)}\lVert f_{k-1}\rVert_\infty \\
&\le D\frac{1+e^{-\lambda}}{1-e^{-\lambda}}\sup_{n\in \Z} \lVert f_n \rVert_\infty,
\end{split}
\]
for $n\in \Z$ and $\mathbf h=(h_n)_{n\in \Z}\in \mathcal Y$. Hence, it follows from~\eqref{b} that 
\[
\sup_{n\in \Z} \lVert (\mathcal T\mathbf h)_n \rVert_\infty <\infty, \quad \text{for $\mathbf y\in \mathcal Y$.}
\]
Since clearly, $(\mathcal T\mathbf h)_n (x) \in S(n)+A_n^{-1}U(n+1)$ for $n\in \Z$ and $x\in X$, we conclude that $\mathcal T$ is well-defined. 

We will now prove that for $c$  sufficiently small, $\mathcal T$ is a contraction on $\mathcal Y$. Indeed, take $\mathbf h^i=(h_n^i)_{n\in \Z} \in \mathcal Y$, $i=1, 2$. By~\eqref{e11}, \eqref{e22} and~\eqref{f}, we have that 
\[
\begin{split}
\lVert (\mathcal T\mathbf h^1)_n (x)-(\mathcal T\mathbf h^2)_n (x)\rVert &\le c\sum_{k=-\infty}^n De^{-\lambda (n-k)}\lVert h_{k-1}^1-h_{k-1}^2\rVert_\infty \\
&\phantom{\le}+c\sum_{k=n+1}^\infty De^{-\lambda (k-n)}\lVert h_{k-1}^1-h_{k-1}^2\rVert_\infty \\
&\le cD\frac{1+e^{-\lambda}}{1-e^{-\lambda}}\lVert \mathbf h^1-\mathbf h^2\rVert_{\mathcal Y}, 
\end{split}
\]
for $x\in X$ and $n\in \Z$. Hence, if 
\[
cD\frac{1+e^{-\lambda}}{1-e^{-\lambda}}<1, 
\]
we have that $\mathcal T$ is a contraction. Therefore, $\mathcal T$ has a unique fixed point $\mathbf h=(h_n)_{n\in \Z}\in \mathcal Y$.  Thus, we have that
\[
\begin{split}
&h_{n+1}(A_n x)  \\
&=(\mathcal T \mathbf h)_{n+1}(A_n x) \\
&=\sum_{k=-\infty}^{n+1} \cA(n+1, k)P_k (f_{k-1}(\cA(k-1, n)x+h_{k-1}(\cA(k-1, n)x))) \\
&\phantom{=}-\sum_{k=n+2}^\infty \cA(n+1, k)(\Id-P_k)(f_{k-1}(\cA(k-1, n)x+h_{k-1}(\cA(k-1, n)x))) \\
&=A_n h_n(x)+f_n(x+h_n(x)),
\end{split}
\]
for $n\in \Z$ and $x\in X$. Setting $H_n=\Id+h_n$ for $n\in \Z$, we see that~\eqref{lin} holds. 

Hence, a fixed point of $\mathcal T$ induces a solution of~\eqref{lin}. The purpose of the following auxiliary lemma is to establish the converse. 

\begin{lemma}\label{lemma: sol is a fixed point}
Let $(G_n)_{n\in \Z}$ be a sequence of continuous  maps on $X$ satisfying the following conditions:
\begin{itemize}
\item for each $n\in \Z$,
\begin{equation}\label{linn}
G_{n+1}\circ A_n=(A_n+f_n)\circ G_n;
\end{equation}
\item  $\sup_{n\in \Z}\lVert G_n-\Id \rVert_\infty <+\infty$;
\item  for each $n\in \Z$ and $x\in X$, \[G_n(x)-x\in S(n)+A_n^{-1}U(n+1).\]
\end{itemize}
 Then, $\mathbf {g}=(g_n)_{n\in \Z}$ is a fixed point of $\mathcal{T}$, where $g_n=G_n-\Id$ for $n\in \Z$.
\end{lemma}
\begin{proof}[Proof of Lemma~\ref{lemma: sol is a fixed point}]
We start by observing that it follows from~\eqref{linn} that
\begin{displaymath}
\begin{split}
g_n(A_{n-1}x)&=(G_n-\Id)(A_{n-1}x)\\
&=G_n(A_{n-1}x)-A_{n-1}x\\
&=(A_{n-1}+f_{n-1})(G_{n-1}(x))-A_{n-1}x\\
&=A_{n-1}(g_{n-1}(x))+f_{n-1}(x+ g_{n-1}(x)),
\end{split}
\end{displaymath}
which implies that
\begin{equation}\label{eq: iteration}
g_n(x)=A_{n-1}(g_{n-1}(A_{n-1}^{-1}x))+f_{n-1}(A_{n-1}^{-1}x+g_{n-1}(A_{n-1}^{-1}x)).
\end{equation}
By iterating~\eqref{eq: iteration}, we conclude that for each $j\leq n$, 
\[
\begin{split}
g_n(x) &=\cA(n, j-1)g_{j-1}(\cA(j-1,n)x) \\
&\phantom{=}+\sum_{k=j}^{n}\cA(n, k)f_{k-1}(\cA(k-1,n)x+ g_{k-1}(\cA(k-1,n)x)).
\end{split}
\]
We now claim that
\begin{equation}\label{1205}
\begin{split}
P_ng_n(x)&=\cA(n, j-1)P_{j-1}g_{j-1}(\cA(j-1,n)x)\\
&\phantom{=}+\sum_{k=j}^{n}\cA(n, k)P_kf_{k-1}(\cA(k-1,n)x+ g_{k-1}(\cA(k-1,n)x)),
\end{split}
\end{equation}
for every $j\leq n$. We begin by proving that 
\begin{equation}\label{eq: Pn}
P_nA_{n-1}g_{n-1}(A_{n-1}^{-1}x)=A_{n-1}P_{n-1}g_{n-1}(A_{n-1}^{-1}x).
\end{equation}
Let us write 
\[
A_{n-1}g_{n-1}(A_{n-1}^{-1}x)=A_{n-1}P_{n-1}g_{n-1}(A_{n-1}^{-1}x)+u_{n}.
\]
Thus, 
\begin{equation}\label{eq: S+U}
g_{n-1}(A_{n-1}^{-1}x)=P_{n-1}g_{n-1}(A_{n-1}^{-1}x)+A_{n-1}^{-1}u_{n}.
\end{equation}
By applying  $P_{n-1}$ on both sides of the above  equality,  we conclude that
\[
P_{n-1}g_{n-1}(A_{n-1}^{-1}x)=P_{n-1}g_{n-1}(A_{n-1}^{-1}x)+P_{n-1}A_{n-1}^{-1}u_{n},
\]
which implies that $P_{n-1}A_{n-1}^{-1}u_{n}=0$. Consequently, $A_{n-1}^{-1}u_{n}\in U(n-1)$.  Therefore, since $g_{n-1}(A_{n-1}^{-1}x)\in S(n-1)+A_{n-1}^{-1}U(n)$, $P_{n-1}g_{n-1}(A_{n-1}^{-1}x)\in S(n-1)$ and $S(n-1)\cap U(n-1)=\{0\}$, we conclude using~\eqref{eq: S+U} that
$u_n\in U(n)$ which easily implies that~\eqref{eq: Pn} holds. 

We now prove~\eqref{1205} by induction. For $j=n$, we have using~\eqref{eq: iteration} and~\eqref{eq: Pn} that 
\begin{equation}\label{606}
\begin{split}
& A_{n-1}P_{n-1}g_{n-1}(A_{n-1}^{-1} x)+P_nf_{n-1}(A_{n-1}^{-1}x+ g_{n-1}(A_{n-1}^{-1}x)) \\
&=P_nA_{n-1}g_{n-1}(A_{n-1}^{-1}x)+P_nf_{n-1}(A_{n-1}^{-1}x+ g_{n-1}(A_{n-1}^{-1}x)) \\
&=P_ng_n(x).
\end{split}
\end{equation}
Assume now that~\eqref{1205} holds for $j$ and we prove that it holds for $j-1$. We have that
\[
\begin{split}
&\sum_{k=j-1}^{n}\cA(n, k)P_kf_{k-1}(\cA(k-1,n)x+ g_{k-1}(\cA(k-1,n)x))\\
&=\sum_{k=j}^{n}\cA(n, k)P_kf_{k-1}(\cA(k-1,n)x+ g_{k-1}(\cA(k-1,n)x)) \\
&\phantom{=}+\cA(n, j-1)P_{j-1}f_{j-2}(\cA(j-2,n)x+ g_{j-2}(\cA(j-2,n)x))\\
&=P_ng_n(x)-\cA(n, j-1)P_{j-1}g_{j-1}(\cA(j-1,n)x)\\
&\phantom{=}+\cA(n, j-1)P_{j-1}f_{j-2}(\cA(j-2,n)x+ g_{j-2}(\cA(j-2,n)x)).
\end{split}
\]
On the other hand, \eqref{606} implies that 
\[
\begin{split}
&-\cA(n, j-2)P_{j-2}g_{j-2}(\cA(j-2, n)x)\\
 &=-\cA(n, j-1)P_{j-1}g_{j-1}(\cA(j-1,n)x)\\
&\phantom{=}+\cA(n, j-1)P_{j-1}f_{j-2}(\cA(j-2,n)x+ g_{j-2}(\cA(j-2,n)x)),
\end{split}
\]
and the desired conclusion follows. 

Using~\eqref{1205} and  since $\sup_{n\in \Z}\lVert G_n-\Id \rVert_\infty <+\infty$ and $g_{j-1}=G_{j-1}-\Id$, it follows that $\|\cA(n, j-1)P_{j-1}g_{j-1}(\cA(j-1,n)x)\| \to 0$ when  $j\to -\infty$. Therefore,
\[
P_ng_n(x)=\sum_{k=-\infty}^{n}\cA(n, k)P_kf_{k-1}(\cA(k-1,n)x+ g_{k-1}(\cA(k-1,n)x)).
\]
Similarly we can prove that
\[
(\Id-P_n)g_n(x)=-\sum_{k=n+1}^\infty \cA(n, k)(\Id-P_k)(f_{k-1}(\cA(k-1, n)x+g_{k-1}(\cA(k-1, n)x))).
\]
By  combining the last two equalities, we conclude that 
\begin{displaymath}
\begin{split}
g_n(x)&=P_ng_n(x)+(\Id-P_n)g_n(x) \\
&=\sum_{k=-\infty}^n \cA(n, k)P_k (f_{k-1}(\cA(k-1, n)x+g_{k-1}(\cA(k-1, n)x))) \\
&\phantom{=}-\sum_{k=n+1}^\infty \cA(n, k)(\Id-P_k)(f_{k-1}(\cA(k-1, n)x+g_{k-1}(\cA(k-1, n)x)))\\
&=(\mathcal T\mathbf g)_n (x).
\end{split}
\end{displaymath}
Hence, 
\begin{displaymath}
\mathcal{T}(\mathbf{g})=\mathbf{g},
\end{displaymath}
and the proof of the lemma is completed. 
\end{proof}

We now define $\mathcal T' \colon \mathcal Y \to \mathcal Y$ by 
\[
\begin{split}
(\mathcal T'\mathbf h)_n (x) 
&=-\sum_{k=-\infty}^n \cA(n, k)P_kf_{k-1}(\mathcal F(k-1, n)x) \\
&\phantom{=}+\sum_{k=n+1}^\infty \cA(n, k)(\Id-P_k)f_{k-1}(\mathcal F(k-1, n)x),
\end{split}
\]
where
\[
\mathcal F(m, n)=\begin{cases}
F_{m-1}\circ \ldots F_n & \text{for $m>n$;}\\
\Id & \text{for $m=n$;}\\
F_{m+1}^{-1}\circ \ldots \circ F_n^{-1} & \text{for $m<n$,}
\end{cases}
\]
and $F_n=A_n+f_n$, $n\in \Z$. 

Again, it follows easily from~\eqref{e11}, \eqref{e22} and~\eqref{b} that $\mathcal T'$ is well-defined. Moreover, we observe that in fact, $\mathcal T'$ a constant map and thus it has  a unique fixed point $\bar{\mathbf h}=(\bar{h}_n)_{n\in \Z}\in \mathcal Y$. We have that 
\[
\begin{split}
&\bar{h}_{n+1}(F_n(x))  \\
&=(\mathcal T' \bar{\mathbf h})_{n+1}(F_n (x)) \\
&=-\sum_{k=-\infty}^{n+1} \cA(n+1, k)P_kf_{k-1}(\mathcal F(k-1, n)x) \\
&\phantom{=}+\sum_{k=n+2}^\infty \cA(n+1, k)(\Id-P_k)f_{k-1}(\mathcal F(k-1, n)x) \\
&=A_n \bar{h}_n(x)-f_n(x),
\end{split}
\]
for $n\in \Z$ and $x\in X$. Setting $\bar{H}_n=\Id+\bar{h} _n$ for $n\in \Z$, we have that
\begin{equation}\label{l2}
\bar{H}_{n+1}\circ F_n=A_n \circ \bar{H}_n, \quad \text{for $n\in \Z$.}
\end{equation}
The following lemma can be proved by arguing as in the proof of Lemma~\ref{lemma: sol is a fixed point}.
\begin{lemma}\label{lemma: sol is a fixed point2}
Let $(R_n)_{n\in \Z}$ be a sequence of continuous  maps on $X$ satisfying the following conditions:
\begin{itemize}
\item for each $n\in \Z$,
\[
R_{n+1}\circ F_n=A_n\circ R_n;
\]
\item  $\sup_{n\in \Z}\lVert R_n-\Id \rVert_\infty <+\infty$;
\item  for each $n\in \Z$ and $x\in X$, \[R_n(x)-x\in S(n)+A_n^{-1}U(n+1).\]
\end{itemize}
 Then, $\mathbf {r}=(r_n)_{n\in \Z}$ is a fixed point of $\mathcal{T'}$, where $r_n=R_n-\Id$ for $n\in \Z$.
\end{lemma}

It follows easily from~\eqref{lin} and~\eqref{l2} that 
\begin{equation}\label{934f}
\bar{H}_{n+1}\circ H_{n+1}\circ A_n=A_n \circ \bar{H}_n\circ H_n, \quad \text{for each $n\in \Z$.}
\end{equation}
We claim that  $\bar{H}_{n}\circ H_n=\Id$ for each $n\in \Z$. Indeed, let $G_n:=\bar{H}_n \circ H_n$ and $g_n:=G_n-\Id$ for $n\in \Z$. Observe that 
\[
g_n(x)=G_n(x)-x=\bar{H}_n(H_n(x))-H_n(x)+H_n(x)-x =\bar{h}_n(H_n(x))+h_n(x),
\]
for $x\in X$ and thus since $\mathbf h, \bar{\mathbf h}\in \mathcal Y$, we have that  $\sup_{n\in \Z} \lVert g_n\rVert_\infty <\infty$. Moreover, using again that $\mathbf h, \bar{\mathbf h}\in \mathcal Y$, we have that $g_n(x)\in S(n)+A_n^{-1}U(n+1)$
for $n\in \Z$ and $x\in X$. We conclude that $\mathbf g=(g_n)_{n\in \Z}=(G_n-\Id)_{n\in \Z}\in \mathcal Y$. It follows from~\eqref{934f} and Lemma~\ref{lemma: sol is a fixed point}  that $\mathbf g$ is a fixed point of $\mathcal T$ in the case when $f_n=0$ for $n\in \Z$. Due to the uniqueness of the fixed point for $\mathcal T$, we can easily conclude that
$g_n=0$ for $n\in \Z$ and thus $G_n=\bar{H}_n\circ H_n=\Id$ for each $n\in \Z$. 
 Similarly, using Lemma~\ref{lemma: sol is a fixed point2} one can easily show that  $H_n \circ \bar{H}_n=\Id$ for each $n\in \Z$ and therefore, 
$H_n$ is a homeomorphism for each $n\in \Z$.  A similar argument also gives us uniqueness of the sequence $(H_n)_{n\in \Z}$, thus completing the proof of the theorem. 
\end{proof}

\begin{remark}
We emphasize that the sequence $(H_n)_{n\in \Z}$ given by  Theorem~\ref{theo: GH} is not unique if we omit the condition that $H_n(x)-x\in S(n)+A_n^{-1}U(n+1)$ for $n\in \Z$ and $x\in X$. Indeed, take a sequence $(A_n)_{n\in \Z}$ that admits a generalized exponential dichotomy but that it doesn't 
admit an exponential dichotomy. Furthermore, set $f_n=0$ for $n\in \Z$. It follows easily from~\cite[Theorem 7.6.5.]{H} that there exists a nonzero sequence $(x_n)_{n\in \Z}\subset X$ such that
\[
\sup_{n\in \Z}\lVert x_n \rVert<\infty \quad \text{and} \quad \text{$x_{n+1}=A_n x_n$ for $n\in \Z$.}
\]
For $n\in \Z$, set
\[
H_n(x)=x+x_n, \quad x\in X.
\]
Then, $(H_n)_{n\in \Z}$ is a sequence of homeomorphisms on $X$ such that~\eqref{lin} holds (recall that $f_n=0$). In addition,  $\sup_{n\in \Z}\lVert H_n-\Id\rVert_\infty <\infty$. However, there exists $n\in \Z$ such that $x_n\neq 0$ and consequently $H_n\neq \Id$.

This example also shows that Lemmas~\ref{lemma: sol is a fixed point} and~\ref{lemma: sol is a fixed point2} were crucial to show that $H_n$ and $H_n'$ constructed in the proof of Theorem~\ref{theo: GH} are inverses of each other. 
\end{remark}

\begin{remark}
We stress that the proof of Theorem~\ref{theo: GH} is heavily inspired by the proof of~\cite[Theorem 1]{B}. Indeed, the space $\mathcal Y$ and the map $\mathcal T$ represent a natural nonautonomous versions of the corresponding objects introduced in the proof of~\cite[Theorem 1]{B}.

\end{remark}

\section{H\"older Conjugacies} \label{M2}
In this section we consider a special class of nonautonomous systems admitting a generalized exponential dichotomy and prove that, restricted to this class, the conjugacies given by Theorem \ref{theo: GH} are H\"older continuous.  Again, this result and its proof are inspired by~\cite[Theorem 3]{B}.

Let $(A_n)_{n\in \Z}$ be a sequence of invertible operators in $\mathcal B(X)$ that admits a generalized exponential dichotomy. Moreover, assume there exist $\rho >0$ and numbers $C_{m,n}\geq 1$ for $m,n\in \Z$ such that 
\begin{equation}\label{eq: tempered1}
\lim_{m\to \pm \infty}\frac{1}{m}\log C_{m,n}=0 \text{ for every } n\in \Z,
\end{equation}
and
\begin{equation}\label{eq: upper grow}
\lVert \cA(m, n)x\rVert \le C_{m,n} e^{\rho |m-n|}\lVert x\rVert  \text{ for every } x\in X \text{ and } m,n\in \Z.
\end{equation}
Finally, we consider $\alpha_0>0$ given by
$
 \alpha_0 =\lambda /  \rho$, where $\lambda >0$ is as in Definition~\ref{def: gexp}. 

We observe that it follows from~\eqref{eq: tempered1} that for each $n\in \Z$ and $\varepsilon>0$, we have that 
\begin{equation}\label{eq: tempered2}
C_n=C_n(\varepsilon):=\sup_{m\in \Z} (C_{m,n}e^{-\varepsilon |m-n|}) <+\infty.
\end{equation}

\begin{theorem}\label{theo: Holder}
Let $\alpha \in (0,\alpha_0)$ and take $\varepsilon>0$ such that $\alpha(\rho+\varepsilon)<\lambda$. Let $C_n=C_n(\epsilon)$ be given by~\eqref{eq: tempered2}.  Moreover, suppose that there exists $c>0$ such that 
\begin{equation}\label{eq: f_n Lips 2}
\|f_n(x)-f_n(y)\|\leq C_{n+1}^{-1}c\|x-y\| \text{ for every } x,y\in X \text{ and } n\in \Z.
\end{equation}
Then, whenever $c>0$ is sufficiently small, the conjugacies $H_n$ and $H_n^{-1}$ given by Theorem \ref{theo: GH} are $\alpha$-H\"older continuous when restricted to any bounded subset of $X$.
\end{theorem}

\begin{remark}
Observe that  since $C_{n,n}\geq 1$ we have that $C_n\ge 1$ for every $n\in \Z$. Consequently,  condition \eqref{eq: f_n Lips 2} implies~\eqref{f}.
\end{remark}

\begin{proof}
By \eqref{b}, we have  that there exists $M>1$ such that $\|f_n(x)\|\le M$ for every $x\in X$ and $n\in \Z$. It follows from~\eqref{f} that 
\begin{equation}\label{eq: f is Holder}
\begin{split}
\|f_n(x)-f_n(y)\|&= \|f_n(x)-f_n(y)\|^{1-\alpha}\|f_n(x)-f_n(y)\|^{\alpha}\\
&\leq 2M\|f_n(x)-f_n(y)\|^{\alpha} \\
&\leq 2M c^\alpha \|x-y\|^\alpha,
\end{split}
\end{equation}
for $n\in \Z$ and $x, y\in X$.

We start  by proving that $H_n^{-1}$ is $\alpha$-H\"older. Recall from the proof of Theorem \ref{theo: GH} that $H_n^{-1}=\text{Id}+\overline{h}_n$, where 
\begin{displaymath}
\begin{split}
 \overline{h}_n (x)&=-\sum_{k=-\infty}^n \cA(n, k)P_kf_{k-1}(\mathcal F(k-1, n)x) \\
&\phantom{=}+\sum_{k=n+1}^\infty \cA(n, k)(\Id-P_k)f_{k-1}(\mathcal F(k-1, n)x).
\end{split}
\end{displaymath}
Given $n\in \Z$ and $x,y\in X$, we have that
\[
\begin{split}
& \|\overline{h}_n (x) -\overline{h}_n (y)\| \\
&= \bigg{\|} -\sum_{k=-\infty}^n \cA(n, k)P_k S^n_{k-1}+\sum_{k=n+1}^\infty \cA(n, k)(\Id-P_k)S^n_{k-1}\bigg{ \|},
\end{split}
\]
where 
\begin{displaymath}
\begin{split}
S_{k-1}^n&=f_{k-1}(\mathcal{F} (k-1, n)x)-f_{k-1}(\mathcal{F} (k-1, n)y).\\
\end{split}
\end{displaymath}
Hence,  \eqref{e11} and \eqref{e22} imply that 
\begin{equation}\label{eq: aux 2 H}
 \|\overline{h}_n (x) -\overline{h}_n (y)\| \le \sum_{k=-\infty}^n D e^{-\lambda (n-k)} \|S^n_{k-1}\|+\sum_{k=n+1}^\infty D e^{-\lambda (k-n)} \|S^n_{k-1} \|.
\end{equation}
On the other hand,  \eqref{eq: f is Holder} implies that 
\begin{equation}\label{cc}
\|S_{k-1}^n\|\leq 2Mc^\alpha\|\mathcal{F} (k-1, n)x-\mathcal{F} (k-1, n)y\|^\alpha.
\end{equation}

Our objective now is to estimate the size of $\|\mathcal{F} (k, n)x-\mathcal{F} (k, n)y\|$. We first introduce certain adapted norms. For each $n\in \Z$ and $x\in X$, let 
\begin{equation*}
\|x\|_n= \sup_{k\in \Z}\left\{ \|\cA(k,n)x\|e^{-(\rho+\varepsilon)|k-n|}\right\}.
\end{equation*}
It follows from~\eqref{eq: upper grow} and~\eqref{eq: tempered2} that
\begin{equation}\label{eq: comp norms}
\|x\|\leq \|x\|_n\leq C_n\|x\|, \text{ for every } x\in X \text{ and }n\in \Z.
\end{equation}
Moreover, 
\begin{equation}\label{eq: upper bound adapted norm}
\begin{split}
\|\cA(m,n)x\|_m&=\sup_{k\in \Z}\left\{ \|\cA(k,m)\cA(m,n)x\|e^{-(\rho+\varepsilon)|k-m|}\right\}\\
&\leq \sup_{k\in \Z}\left\{ \|\cA(k,n)x\|e^{-(\rho+\varepsilon)|k-n|} e^{(\rho+\varepsilon)|m-n|} \right\}\\
&\leq  e^{(\rho+\varepsilon)|m-n|} \|x\|_n,
\end{split}
\end{equation}
for every $x\in X$ and $m,n\in\Z$. Furthermore, observe that condition \eqref{eq: f_n Lips 2} combined with \eqref{eq: comp norms} implies that
\begin{equation}\label{eq: f_n Lips adapted}
\begin{split}
\|f_n(x)-f_n(y)\|_{n+1}&\leq C_{n+1} \|f_n(x)-f_n(y)\|\\
&\leq C_{n+1} C_{n+1}^{-1}c\|x-y\|\\
&\leq c \|x-y\|_k,
\end{split}
\end{equation}
for every $x,y\in X$ and $n,k\in \Z$.

We now claim that for every $k\geq n$ and $x, y\in X$, we have that 
\begin{equation}\label{c1}
\|\mathcal{F} (k, n)x-\mathcal{F} (k, n)y\| \leq C_n (e^{\rho+\varepsilon}+c)^{k-n}\|x-y\|.
\end{equation}
Indeed, using \eqref{eq: comp norms}, \eqref{eq: upper bound adapted norm} and \eqref{eq: f_n Lips adapted} we get that
\begin{displaymath}
\begin{split}
\|\mathcal{F} (n+1, n)x-\mathcal{F} (n+1, n)y\|_{n+1}&=\|(A_n+f_n)x-(A_n+f_n)y\|_{n+1}\\
&\leq \|A_n(x-y)\|_{n+1}+\|f_n(x)-f_n(x)\|_{n+1} \\
&\leq (e^{\rho+\varepsilon}+c)\|x-y\|_{n}.
\end{split}
\end{displaymath}
Thus, proceeding by induction we conclude for every $k\geq n$ and $x, y\in X$,
\begin{displaymath}
\begin{split}
\|\mathcal{F} (k, n)x-\mathcal{F} (k, n)y\|_{k}&\leq (e^{\rho+\varepsilon}+c)^{k-n}\|x-y\|_{n}.
\end{split}
\end{displaymath}
Consequently, using \eqref{eq: comp norms} we conclude that \eqref{c1} holds.
Similarly, for every $c>0$ small enough we claim that
\begin{equation}\label{c2}
\|\mathcal{F} (k, n)x-\mathcal{F} (k, n)y\| \leq C_n \left(\frac{e^{\rho+\varepsilon}}{1-ce^{\rho+\varepsilon}}\right)^{n-k}\|x-y\,
\end{equation}
for every $k\leq n$ and $x, y\in X$. In fact, recalling that $F_j=A_j+f_j$, we have that $F_j^{-1}(x)=A_j^{-1}x-A_j^{-1}(f_j(F_ j^{-1}(x)))$. Consequently, 
\begin{displaymath}
\begin{split}
\|F_j^{-1}(x)-F_j^{-1}(y)\|_{j}&\leq\|A_j^{-1}x-A_j^{-1}y\|_j+\|A_j^{-1}(f_j(F_ j^{-1}x))-A_j^{-1}(f_j(F_ j^{-1}y))\|_{j} \\
&\leq e^{\rho+\varepsilon} \|x-y\|_{j+1}+ e^{\rho+\varepsilon} \|f_j(F_ j^{-1}x)-f_j(F_ j^{-1}y)\|_{j+1}\\
& \leq  e^{\rho+\varepsilon}\|x-y\|_{j+1}+ce^{\rho+\varepsilon}\|F_ j^{-1}(x)-F_ j^{-1}(y)\|_j.
\end{split}
\end{displaymath}
Therefore, whenever $ce^{\rho+\varepsilon} <1$, we have that 
\begin{displaymath}
\|F_j^{-1}(x)-F_j^{-1}(y)\|_{j} \leq \frac{e^{\rho+\varepsilon}}{1-ce^{\rho+\varepsilon}}\|x-y\|_{j+1}.
\end{displaymath}
Again,  proceeding by induction and using \eqref{eq: comp norms} we conclude that~\eqref{c2} holds. Hence, it follows from~\eqref{eq: aux 2 H}, \eqref{cc}, \eqref{c1} and~\eqref{c2} that
\begin{displaymath}
\begin{split}
\|\overline{h}_n (x) -\overline{h}_n (y)\|&\leq 
2MDC_n^\alpha c^\alpha \sum_{k=-\infty}^{n}  e^{-\lambda (n-k)} (e^{\rho+\varepsilon} +c)^{\alpha(n-(k-1))} \|x-y \|^\alpha\\
&\phantom{=} +2MDC_n^\alpha c^\alpha \sum_{k=n+1}^{+\infty}  e^{-\lambda (k-n)} \left(\frac{e^{\rho+\varepsilon}}{1-ce^{\rho+\varepsilon}}\right)^{\alpha((k-1)-n)} \|x-y \|^\alpha\\
&= C_n^\alpha L\|x-y\|^\alpha,
\end{split}
\end{displaymath} 
where 
\begin{displaymath}
L=2MD c^\alpha\left( \sum_{k=-\infty}^{n}  e^{-\lambda (n-k)} (e^{\rho+\varepsilon} +c)^{\alpha(n-k+1)} +\sum_{k=n+1}^{+\infty}  e^{-\lambda (k-n)} \left(\frac{e^{\rho+\varepsilon}}{1-ce^{\rho+\varepsilon}}\right)^{\alpha (k-1-n)} \right).
\end{displaymath} 
Since $\alpha (\rho+\varepsilon) <\lambda$,  we have that  $L<+\infty$  provided that $c$ is sufficiently small. 
 Consequently, considering $L_n:=C_n^\alpha L$, we have that 
\begin{displaymath}
\begin{split}
\|H^{-1}_n(x)-H^{-1}_n(y)\|&= \|x+\overline{h}_n(x)-y-\overline{h}_n(y)\|\\
&\leq \|x-y\|+\|h_n(x)-h_n(y)\|\\
&\leq \left( \|x-y\|^{1-\alpha}+L_n\right)\|x-y\|^\alpha,
\end{split}
\end{displaymath}
for any $x,y\in X$ which implies that $H_n^{-1}$ is $\alpha$-H\"older continuous when restricted to any bounded subset of $X$ as claimed.

We now prove that $H_n$ is $\alpha$-H\"older continuous (when restricted to any bounded subset of $X$). Let $(\mathcal{Y},\|\cdot\|_\mathcal{Y})$ and  $\mathcal{T}:\mathcal{Y}\to \mathcal{Y}$ be as  in the proof of Theorem \ref{theo: GH}. Given $K>1$, we denote by $\mathcal{Y}_{\alpha,K}$ the subset of 
$\mathcal Y$ which consists of all $\mathbf{h}=(h_n)_{n\in \Z}\in \mathcal{Y}$ such that for every $x,y\in X$ and $n\in \Z$,
\begin{equation}\label{eq: lips}
\|h_n(x)-h_n(y)\|\leq K\|x-y\|_n^\alpha.
\end{equation}

We claim now that $\mathcal{T}(\mathcal{Y}_{\alpha,K})\subset \mathcal{Y}_{\alpha,K}$ whenever $c$ is sufficiently small. Indeed, given $\mathbf{h}=(h_n)_{n\in \Z}\in \mathcal{Y}_{\alpha,K}$, $n\in \Z$ and $x,y\in X$, we have that 
\[
\begin{split}
&\|(\mathcal T\mathbf h)_n (x) -(\mathcal T\mathbf h)_n (y)\| \\
&=\bigg{\|} \sum_{k=-\infty}^n \cA(n, k)P_k T^n_{k-1}-\sum_{k=n+1}^\infty \cA(n, k)(\Id-P_k)T^n_{k-1} \bigg{\|},
\end{split}
\]
where 
\begin{displaymath}
\begin{split}
T_{k-1}^n&=f_{k-1}(\cA(k-1, n)x+h_{k-1}(\cA(k-1, n)x)))\\
& \phantom{=} -f_{k-1}(\cA(k-1, n)y+h_{k-1}(\cA(k-1, n)y))).
\end{split}
\end{displaymath}
Hence, it follows from~\eqref{e11} and \eqref{e22} that 
\begin{equation}\label{eq: aux 1}
\|(\mathcal T\mathbf h)_n (x) -(\mathcal T\mathbf h)_n (y)\| \le \sum_{k=-\infty}^n D e^{-\lambda (n-k)} \|T^n_{k-1}\|+\sum_{k=n+1}^\infty D e^{-\lambda (k-n)} \|T^n_{k-1} \|,
\end{equation}
for $x, y\in X$ and $n\in \Z$. 
Next, we want to  estimate $\|T_{k-1}^n\|$ for $k\leq n$. Assume first that $\|\mathcal{A}(k-1,n)(x-y)\|_{k-1}\geq 1$. Then, using \eqref{eq: f is Holder}, \eqref{eq: comp norms} and \eqref{eq: upper bound adapted norm} we obtain that
\begin{displaymath}
\begin{split}
\|T_{k-1}^n\|&\leq 2Mc^\alpha \left(\|(\cA(k-1, n)(x-y)\|+\|h_{k-1}(\cA(k-1, n)x)))-h_{k-1}(\cA(k-1, n)y))\|\right)^\alpha \\
&\leq 2Mc^\alpha \left(\|(\cA(k-1, n)(x-y)\|_{k-1}+K\|\cA(k-1, n)(x-y)\|_{k-1}^\alpha\right)^\alpha \\
&\leq 2Mc^\alpha \left(e^{(\rho+\varepsilon) (n-k+1)}\|x-y\|_n+K\|\cA(k-1, n)(x-y)\|_{k-1}\right)^\alpha \\
&\leq 2Mc^\alpha \left(e^{(\rho+\varepsilon) (n-k+1)}\|x-y\|_n+Ke^{(\rho+\varepsilon) (n-k+1)}\|x-y\|_n\right)^\alpha\\
&\leq 4MKc^\alpha e^{\alpha(\rho+\varepsilon) (n-k+1)}\|x-y\|_n^\alpha.
\end{split}
\end{displaymath}
On the other hand, if $\|\mathcal{A}(k-1,n)(x-y)\|_{k-1}< 1$ then, using~\eqref{f}, \eqref{eq: comp norms} and \eqref{eq: upper bound adapted norm} we have that
\begin{displaymath}
\begin{split}
\|T_{k-1}^n\|&\leq c \left(\|(\cA(k-1, n)(x-y)\|+\|h_{k-1}(\cA(k-1, n)x)))-h_{k-1}(\cA(k-1, n)y))\|\right) \\
&\leq c\left(\|(\cA(k-1, n)(x-y)\|_{k-1}+K\|\cA(k-1, n)(x-y)\|_{k-1}^\alpha\right)\\
&\leq c\left(\|(\cA(k-1, n)(x-y)\|_{k-1}^\alpha+K\|\cA(k-1, n)(x-y)\|_{k-1}^\alpha\right)\\
&\leq c \left( e^{\alpha(\rho+\varepsilon) (n-k+1)}\|x-y\|_n^\alpha+Ke^{\alpha(\rho+\varepsilon) (n-k+1)}\|x-y\|_n^\alpha\right)\\
&\leq 2Kc e^{\alpha(\rho+\varepsilon) (n-k+1)}\|x-y\|_n^\alpha\\
&\leq 4MKc e^{\alpha(\rho+\varepsilon) (n-k+1)}\|x-y\|_n^\alpha.
\end{split}
\end{displaymath}
Thus, assuming $c\in (0,1]$, we have that
\begin{displaymath}
\|T_{k-1}^n\|\leq 4MKc^\alpha e^{\alpha(\rho+\varepsilon) (n-k+1)}\|x-y\|_n^\alpha
\end{displaymath}
for any $k\leq n$ and $x, y\in X$. Similarly, for $k>n$ and $x, y\in X$,  we have that 
\begin{displaymath}
\|T_{k-1}^n\|\leq 4MKc^\alpha e^{\alpha(\rho+\varepsilon) (k-n-1)}\|x-y\|_n^\alpha.
\end{displaymath}
 By plugging the last two inequalities into \eqref{eq: aux 1} and recalling that $\alpha(\rho+\varepsilon)<\lambda$, we obtain that
\begin{displaymath}
\begin{split}
\|(\mathcal T\mathbf h)_n (x) -(\mathcal T\mathbf h)_n (x)\|&\leq 4MKDc^\alpha e^{\alpha (\rho+\varepsilon)}\sum_{k=-\infty}^{+\infty} e^{(-\lambda+\alpha (\rho+\varepsilon)) |k|}\|x-y\|_n^\alpha \\
&\leq 4MKDc^\alpha e^{\alpha (\rho+\varepsilon)} \frac{1+e^{-\lambda+\alpha (\rho+\varepsilon)}}{1-e^{-\lambda+\alpha(\rho+\varepsilon)}}\|x-y\|_n^\alpha.
\end{split}
\end{displaymath}
Thus, taking $c>0$ sufficiently small so that \[4MDc^\alpha e^{\alpha (\rho+\varepsilon)} \frac{1+e^{-\lambda+\alpha (\rho+\varepsilon)}}{1-e^{-\lambda+\alpha (\rho+\varepsilon)}}\le 1, \] it follows that $\mathcal{T}(\mathcal{Y}_{\alpha,K})\subset \mathcal{Y}_{\alpha,K}$ as claimed. Therefore, observing that $\mathcal{Y}_{\alpha,K}$ is a closed subset of $(\mathcal{Y},\|\cdot\|_\mathcal{Y})$ and recalling that $\mathcal{T}:\mathcal{Y}\to \mathcal{Y}$ is a contraction, we have that the unique fixed point $\mathbf{h}=(h_n)_{n\in \Z}$ of $\mathcal{T}$ satisfies $\mathbf{h}\in \mathcal{Y}_{\alpha,K}$. Thus, since the conjugacy given by Theorem \ref{theo: GH} is of the form $H_n=\text{Id}+h_n$ and using \eqref{eq: comp norms} it follows that 
\begin{displaymath}
\begin{split}
\|H_n(x)-H_n(y)\|&= \|x+h_n(x)-y-h_n(y)\|\\
&\leq \|x-y\|+\|h_n(x)-h_n(y)\|\\
&\leq \left( \|x-y\|^{1-\alpha}+K\right)\|x-y\|_n^\alpha\\
&\leq C_n^\alpha\left( \|x-y\|^{1-\alpha}+K\right)\|x-y\|^\alpha
\end{split}
\end{displaymath}
for any $x,y\in X$, which implies that $H_n$ is $\alpha$-H\"older continuous when restricted to any bounded subset of $X$. The proof of the theorem is completed. 
\end{proof}

\begin{remark}\label{rem: bounded case}
We observe that in the particular case when \[ \sup_{n\in \Z} \max \{\lVert A_n\rVert, \lVert A_n^{-1}\rVert \}<\infty, \] Theorems \ref{theo: GH} and \ref{theo: Holder} can be deduced from the results in \cite{B}. Indeed, let $X_\infty\subset X^\Z$ be given by \[ X_\infty=\bigg \{(x_n)_{n\in \Z}\in X^\Z; \ \sup_{n\in \Z}\|x_n\|<+\infty \bigg \}.\] It is easy to see that $X_\infty$ is a Banach space when endowed with the norm $\|(x_n)_{n\in \Z}\|_\infty
=\sup_{n\in \Z}\|x_n\|$. Consider $T:X_\infty \to X_\infty $ given by 
\begin{displaymath}
T((x_n)_{n\in \Z})=((A_{n-1}x_{n-1})_{n\in \Z}).
\end{displaymath}
It is easy to check that $T$ is a well defined, bounded and invertible operator. Moreover, one can easily verify that $T$ is a generalized hyperbolic operator  (see  Example \ref{ex: ghyp op}).  In addition, we can define $f\colon X_\infty \to X_\infty$ by
\[
f((x_n)_{n\in \Z})=(f_{n-1}(x_{n-1}))_{n\in \Z}.
\]
By~\eqref{b}, we have that $f$ is well-defined. Moreover,  \eqref{f} implies that $f$ is Lipschitz with the  Lipschitz constant bounded by $c$. By~\cite[Theorem 1]{B}, for $c$ sufficiently small, there exists  a homeomorphism $H\colon Y_\infty \to Y_\infty$ such that 
$H\circ T=(T+f)\circ H$ and 
\[
\sup \{ \| H((x_n)_{n\in \Z})- (x_n)_{n\in \Z} \|_\infty: (x_n)_{n\in \Z} \in Y_\infty \} <+\infty. 
\]
It is not difficult to construct conjugacies $H_n$ as in the statement of Theorem~\ref{theo: GH} directly from $H$. The details are left to the reader. We refer to~\cite[Section 3]{BDV} for a similar approach in a different setting

\end{remark}

\section{Examples}\label{sec: examples}
In this section we present several examples of nonautonomous systems admitting a generalized exponential dichotomy focusing  on those that don't admit an  exponential dichotomy.  We stress that Examples \ref{ex: ghyp op}, \ref{ex: ex2} and \ref{ex: ex3} are of a more abstract nature while Example \ref{ex: unbounded} is somewhat more concrete.

\begin{example}[Generalized Hyperbolic operators] \label{ex: ghyp op}
Let $T\in \mathcal B(X)$ be an invertible map. We say that $T$ is a \emph{generalized hyperbolic operator} if there are closed subspaces $E^s$ and $E^u$ satisfying $X= E^s\oplus E^u$ such that 
\begin{displaymath}
T(E^s)\subset E^s \text{ and } T^{-1}(E^u)\subset E^u
\end{displaymath}
and, moreover, $T_{|E^s}$ and $T^{-1}_{|E^u}$ are uniform contractions. These last two requirements are equivalent to
\begin{equation*}
\sigma(T_{|E^s})\subset \mathbb{D} \text{ and }  \sigma(T^{-1}_{|E^u})\subset \mathbb{D},
\end{equation*}
where $\mathbb{D}=\{z\in \mathbb C: \| z\| <1\}$. Obviously, given a generalized hyperbolic operator $T$, the sequence $(A_n)_{n\in \Z}$ given by $A_n=T$, $n\in \Z$ admits a generalized exponential dichotomy. Furthermore, if $T$ is not hyperbolic, $(A_n)_{n\in \Z}$ doesn't admit an exponential dichotomy.

We now present some explicits examples of generalized hyperbolic operators that are not necessarily hyperbolic. We are going to use these examples in the constructions below.

\emph{Weighted shifts.} Let $X=l_p(\Z)$ for $1\leq p<+\infty$ or $X=c_0(\Z)$ and let $\omega=(\omega_n)_{n\in \Z}$ be a bounded sequence of numbers satisfying $\inf_{n\in \Z}|\omega_n|>0$. We consider the bilateral weighted left shift $S_\omega:X\to X$ given by
\begin{displaymath}
S_\omega((x_n)_{n\in \Z})=(\omega_{n+1}x_{n+1})_{n\in \Z}, \quad (x_n)_{n\in \Z} \in X.
\end{displaymath} 
Observe that boundedness of $\omega$ is a necessary and sufficient condition for $S_\omega$ to be a well-defined operator in $X$ while condition $\inf_{n\in \Z}|\omega_n|>0$ implies that $S_\omega$ is invertible. Suppose moreover that
\begin{displaymath}
\limsup_{n\to \infty}\sup_{k\in \N}|\omega_{-k}\omega_{-k-1}\cdot\ldots\cdot\omega_{-k-n}|^{\frac{1}{n}}<1 \text{ and } \liminf_{n \to \infty} \inf_{k\in \N} |\omega_{k}\omega_{k+1}\cdot\ldots\cdot\omega_{k+n}|^{\frac{1}{n}}>1.
\end{displaymath}
Thus, considering
\begin{displaymath}
E^s=\{(x_n)_{n\in \Z} \in X; x_n=0 \text{ for every } n>0 \}
\end{displaymath}
and
\begin{displaymath}
E^{u}=\{(x_n)_{n\in \Z}\in X; x_n=0\text{ for every } n\leq 0\},
\end{displaymath}
and using the spectral radius formula one can easily see that this is an example of generalized hyperbolic operator. Moreover, it was proved in \cite{B2} (see also \cite[Theorem B]{B}) that it is not hyperbolic. Simple examples of sequences $\omega$ satisfying the previous conditions are given, for instance, whenever
\begin{displaymath}
\lambda^{-1}< \omega_n< \sigma \text{ for } n<0 \text{ and } \sigma^{-1}< \omega_n< \lambda \text{ for } n>0
\end{displaymath}
where $\lambda<1<\sigma$.

\emph{Operators in $L^2(\mathbb{R})$.} Take $\gamma_0>0$ and let $\gamma \colon \R \to \R$ be such that $\gamma(x)>\gamma_0$ for $x\le 0$ and $\gamma(x)<-\gamma_0$ for $x>0$. For each $t_0>0$, we define a bounded linear operator $T_{t_0}$ on $L^2(\R)$ by 
\begin{displaymath}
\left[T_{t_0}\psi\right](x)=\lambda_{t_0}(x)\psi(x-t_0) \quad \text{for $\psi \in L^2(\R)$,}
\end{displaymath}
where 
\begin{displaymath}
\lambda_{t_0}(x)= e^{\int_0^{t_0} \gamma(x-s) ds}.
\end{displaymath}
 Observing that $L^2(\mathbb{R})=E^s\oplus E^u$, where $E^s=\{\psi \in L^2(\R): \psi(x)=0 \text{ for }  x<0\}$ and $E^u=\{\psi \in L^2(\R): \psi(x)=0  \text{ for } x>0\}$, it follows that $T_{t_0}$ for $t_0>0$ is a generalized hyperbolic operator. This example is taken from~\cite[Section 3]{C} where the reader can also find several other examples.
\end{example}

We now use these classes of operators to construct examples of nonautonomous dynamics that admits a generalized exponential dichotomy. 

\begin{example} \label{ex: ex2}
Let $\{T_1,T_2,\ldots,T_k\}$ be a finite family of generalized hyperbolic operators acting on $X$. In particular, for every $i=1,\ldots,k$ there are constants $D_i>0$ and $\lambda_i>0$ and a decomposition $X=E_{T_i}^s\oplus E^u_{T_i}$ into closed subspaces so that
\begin{displaymath}
\|T^n_ix\|\leq D_ie^{-\lambda_i n}\|x\| \text{ for every } x\in E^s_{T_i} \text{ and } n\geq 0
\end{displaymath}
and
\begin{displaymath}
\|T^n_ix\|\leq D_ie^{-\lambda_i |n|}\|x\| \text{ for every } x\in E^u_{T_i} \text{ and } n\leq 0.
\end{displaymath}
Suppose, moreover, that
\begin{displaymath}
E^s_{T_i}=E^s_{T_j} \text{ and } E^u_{T_i}=E^u_{T_j}
\end{displaymath} 
for every $i,j\in \{1,2,\ldots,k\}$ and $D_i=1$ for every $i\in \{1,2,\ldots,k\}$. Let us now consider any sequence $(A_n)_{n\in \Z}$ of operators such that $A_n\in \{T_1,T_2,\ldots,T_k\}$ for every $n\in \Z$. Then, it is easy to see that $(A_n)_{n\in \Z}$ admits a generalized exponential dichotomy with $S(n)=E^s_{T_i}$ and $U(n)=E^u_{T_i}$ for every $n\in \Z$, $D=1$ and $\lambda=\min\{\lambda_1,\ldots,\lambda_k\}>0$. Moreover, if $A_n\neq A_m$ for some $n,m\in \Z$ then the system is actually nonautonomous and, furthermore, whenever some of the $A_n$'s is not hyperbolic the sequence does not admit an exponential dichotomy. Examples of families of operators satisfying these hypothesis are the weighted shifts and the operators in $L^2(\R)$ presented above.
\end{example}

\begin{example}\label{ex: ex3}
As in the previous example, let $\{T_1,T_2,\ldots,T_k\}$ be a finite family of generalized hyperbolic operators acting on $X$ and satisfying $E^s_{T_i}=E^s_{T_j}$ and $E^u_{T_i}=E^u_{T_j}$ for every $i,j\in \{1,2,\ldots,k\}$. Denote these common subspaces by $E^s$ and $E^u$, respectively, consider $\tilde{\lambda}=\min \{\lambda_1,\ldots,\lambda_k\}>0$ and assume $D_i=1$ for every $i$ as before. Let $U\in \mathcal{B}(X)$ be an invertible operator satisfying
\begin{displaymath}
U(E^s)\subset E^s, U^{-1}(E^u)\subset E^u \text{ and } U(E^u)\cap E^s\neq \emptyset
\end{displaymath}
so that 
\begin{equation} \label{eq: norm U example}
\|U\|<e^{\tilde{\lambda}} \text{ and } \|U^{-1}\|<e^{\tilde{\lambda}}.
\end{equation}
Let $(A_n)_{n\in \Z}$ be any sequence of operators with $A_n\in \{U, T_1,T_2,\ldots,T_k\}$ for every $n\in \Z$ so that the operator $U$ never appear in pairs, that is, if $A_n=U$ then $A_{n+1}\neq U$ and $A_{n-1}\neq U$. Thus, observing that for every $x\in E^s$ and $i\in \{1,2,\ldots k\}$,
\begin{displaymath}
\|UT_ix\|\leq \|U\|\|T_ix\|\leq \|U\|e^{-\tilde{\lambda}}\|x\| 
\end{displaymath}
and
\begin{displaymath}
\|T_iUx\|\leq e^{-\tilde{\lambda}}\|Ux\|\leq e^{-\tilde{\lambda}}\|U\|\|x\|,
\end{displaymath}
since $U(E^s)\subset E^s$, and similarly for $\|U^{-1}T_i^{-1}x\|$ and $\|T_i^{-1}U^{-1}x\|$ for every $x\in E^u$, it follows that $(A_n)_{n\in \Z}$ admits a generalized exponential dichotomy with $S(n)=E^s$ and $U(n)=E^u$ for every $n\in \Z$, $D=1$ and $\lambda=\min\{\tilde{\lambda} -\log \|U\|, \tilde{\lambda} - \log \|U^{-1}\|\}>0$. Observe that if $A_n=U$ for some $n\in \Z$ then the sequence does not admit an exponential dichotomy even when all operators $T_i$ are hyperbolic. Moreover, as before, this construction, in general, gives rise to nonautonomous systems. Furthermore, this construction can be obviously generalized: instead of taking just one operator $U$ as above one can take several; we can allow the $U$'s to appear in pairs, triples and so on by adding some more restrictive hypothesis on its norm; the assumption that $D_i=1$ for every $i\in \{1,2,\ldots,k\}$ can be removed by changing hypothesis \eqref{eq: norm U example} by $\|U_{|E^s}\|\|T_{i|E^s}\|<1$ and $\|(U^{-1})_{|E^u}\|\|(T^{-1}_{i})_{|E^u}\|<1$ for every $i\in \{1,2,\ldots,k\}$.
\end{example}

Our next example shows that the  concept of a generalized exponential dichotomy is rather flexible. Moreover, it gives us an example to which our results apply that does not fit in the particular settings discussed in Remarks \ref{rem: A+f homeo} and \ref{rem: bounded case}.
\begin{example} \label{ex: unbounded}
Let $X=l_p(\Z)$ for $1\leq p\leq +\infty$ or $X=c_0(\Z)$. Take $\lambda >0$ and consider a sequence $(\lambda_n)_{n\in\Z}$ with the property that  $\lambda_n \ge e^\lambda$ for every $n\in \Z$. For each $n\in \mathbb{N}$, let $A_n:X\to X$ be given by
\begin{displaymath}
A_n((x_k)_{k\in \Z})=(y_k)_{k\in \Z}, 
\end{displaymath}
where
\[
y_k=\begin{cases}
\lambda_n^{-1} x_k & \text{for $|k|\leq n$;} \\
\lambda_n x_k& \text{for $|k|> n$.}
\end{cases}
\]
Moreover, for $n>0$  let $A_{-n}:X\to X$ be given by
\begin{displaymath}
A_{-n}((x_k)_{k\in \Z})=(\lambda_{-n} x_k)_{k\in \Z}.
\end{displaymath}
It is easy to see that each $A_n$ is invertible. Moreover, \[ \max \{\|A_n\|,\|A_n^{-1}\|\}=\lambda_n \quad \text{for every $n\in \Z$.} \] In particular, if we choose $\lambda_n$, $n\in \Z$, so that $\sup_{n\in \Z}\{\lambda_n\}=+\infty$, we have that 
 \[ \sup_{n\in \Z} \max \{\|A_n\|,\|A_n^{-1}\|=+\infty.\]  For $n\in \N$, set 
\begin{displaymath}
S(n)=\{(x_k)_{k\in \Z} \in X: x_k=0 \text{ for every } |k|>n \}
\end{displaymath}
and
\begin{displaymath}
U(n)=\{(x_k)_{k\in \Z}\in X: x_k=0\text{ for every } |k|\leq n\}.
\end{displaymath}
Furthermore, for $n>0$ let 
\begin{displaymath}
S(-n)=\{(0)_{k\in \Z} \in X \} \text{ and } U(-n)=X.
\end{displaymath}
Clearly, 
\[
S(n)\oplus U(n)=X, \quad \text{for $n\in \Z$.}
\]
Moreover,  $S(n)\subset S(n+1)$ and $U(n+1)\subset  U(n)$ for every $n\in \mathbb{Z}$. Thus, since $A_n(S(n))=S(n)$ and $A_n^{-1}(U(n+1))=U(n+1)$ for every $n\in \Z$, it follows that $A_n(S(n))\subset S(n+1)$ and $A_n^{-1}(U(n+1))\subset U(n)$ for every $n\in \Z$. Finally, we observe that  for every $n\in \Z$ and $x\in S(n)$ we have that $\|A_nx\|=\lambda_n^{-1}\|x\|\leq e^{-\lambda}\|x\|$,  while for $x\in U(n)$ we have that $\|A_ n^{-1}x\|=\lambda_n^{-1}\|x\|\leq e^{-\lambda}\|x\|$. We conclude that the sequence $(A_n)_{n\in \Z}$ admits a generalized exponential dichotomy. Moreover, it is easy to see that $(A_n)_{n\in \Z}$ doesn't admit an exponential dichotomy. 

Let $c\in (0,1)$ be given by Theorem \ref{theo: GH} associated to $(A_n)_{n\in \Z}$ and consider a sequence of maps $f_n:X\to X$ such that $\sup_{n\in \Z}\|f_n\|_\infty<+\infty$ and
\begin{equation*}
\|f_n(x)-f_n(y)\|\leq c_n\|x-y\| \text{ for every } x,y\in X, 
\end{equation*}
with $c_n\leq c\lambda_n^{-1}$ for every $n\in \Z$. It is easy to see that $F_n=A_n+f_n$ is a homeomorphism for every $n\in \Z$. Indeed, for $x, y\in X$ we have that 
\begin{displaymath}
\begin{split}
\|A_nx+f_n(x)-A_ny-f_n(y)\|&\geq  \|A_nx-A_ny\|-\|f_n(x)-f_n(y)\|\\
&\geq \lambda_n^{-1}\|x-y\|-c_n\|x-y\|\\
&\geq\lambda_n^{-1}(1-c)\|x-y\|,
\end{split}
\end{displaymath}
which proves that  $F_n$ is injective. In order to prove that $F_n$ is surjective, take $y\in X$ and consider $H:X\to X$ given by $H(x)=A_n^{-1}y-A_n^{-1}f_n(x)$, $x\in X$. Hence, 
\begin{displaymath}
\begin{split}
\|H(x)-H(z)\|&=\|A_n^{-1}y-A_n^{-1}f_n(x)-A_n^{-1}y+A_n^{-1}f_n(z)\|\\
&=\|A_n^{-1}f_n(x)-A_n^{-1}f_n(z)\|\\
&\leq \lambda_n \|f_n(x)-f_n(z)\|\\
&\leq \lambda_n c_n\|x-z\|\\
&\leq c\|x-z\|,
\end{split}
\end{displaymath}
for every $x,z\in X$. Since $c<1$, $H$ is a contraction. Therefore, there exists $x\in X$ such that $H(x)=x$. Consequently, $x=A_n^{-1}y-A_n^{-1}f_n(x)$ which implies that $A_nx+f_n(x)=y$. We conclude that  $F_n$ is surjective. The fact that $F_n$ is continuous is obvious while the fact that $F_n^{-1}$ is continuous (actually Lipschitz) follows from an argument similar to the one used to show inequality \eqref{c2}. Thus, $A_n+f_n$ are homeomorphisms for every $n\in \Z$ and we may apply Theorem \ref{theo: GH} to it. 
 Finally, one can easily choose sequences $(\lambda_n)_{n\in \Z}$ satisfying $\sup_{n\in \Z}\lambda_n=+\infty$ for which we still can apply Theorem \ref{theo: Holder}. For instance, given $n\in \Z$, let $\lambda_n$ be such that
\begin{displaymath}
\lambda_n=\begin{cases}
|n| & \text{if $|n|=10^k$ for some $k\in \mathbb{N}\setminus \{0\}$;}\\
$e$ & \text{otherwise.} \\
\end{cases}
\end{displaymath}
Then, it is easy to see that \eqref{eq: upper grow} is satisfied with $C_{m,n}=|m|+|n|+1$ and $\rho =1$ and, consequently, Theorem \ref{theo: Holder} may be applied to $(A_n+f_n)_{n\in \Z}$.
\end{example}


\medskip{\bf Acknowledgements.}
We would like to thank to the referee for several suggestions that helped us to improve the quality of our paper. Moreover, we thank him/her for pointing out to us Remark \ref{rem: bounded case}. We would also like to express our gratitude to Ken Palmer who read the paper and noticed a gap in the first version. 
 L.B. was partially supported by a CNPq-Brazil PQ fellowship under Grant No. 306484/2018-8. D. D. was supported in part by Croatian Science Foundation under the project
IP-2019-04-1239 and by the University of Rijeka under the projects uniri-prirod-18-9
and uniri-prprirod-19-16.


\bibliographystyle{abbrv}

\end{document}